\documentclass[12pt]{amsart}

\usepackage{amsmath,amsfonts,amsthm,amsopn,cite,mathrsfs}
\usepackage{epsfig,verbatim}
\usepackage{subfigure}
\newtheorem{tm}{Theorem}[section]    
\newtheorem{lemma}[tm]{Lemma}
\newtheorem{prop}[tm]{Proposition}
\newtheorem{cor}[tm]{Corollary}
\newtheorem{definition}[tm]{Definition}

\newcommand{\field}[1]{\mathbb{#1}}
\newcommand{\bR}{\field{R}}        
\newcommand{\bN}{\field{N}}        
\newcommand{\bZ}{\field{Z}}        
\newcommand{\bC}{\field{C}}        
\newcommand{\bT}{\field{T}}        %
\def\cO{\mathcal{ O}}

\def\cP{\mathcal{ P}}

\def\cT{\mathcal{ T}}
\def\cX{\mathcal{ X}}

\def\fhat{\hat{f}}

\def\rd{\bR^d}

\def\<{\left<}
\def\>{\right>}

\def\inv{^{-1}}

\newcommand{\mz}{Marcinkie\-wicz-Zygmund}
\newcommand{\xkn}{x_{n,k}}
\newcommand{\tkn}{\tau_{n,k}}
\newcommand{\hs}{H^\sigma}
\newcommand{\scrit}{\sigma _{\textrm{crit}}}

\begin{document}
\begin{abstract}
Given  a sequence of \mz\ inequalities in $L^2$, we derive
approximation theorems and quadrature rules. The derivation is
completely elementary and requires only the definition of \mz\
inequality, Sobolev spaces, and the solution of  least square problems.
\end{abstract}

\title[Marcinkiewicz-Zygmund Inequalities]{Sampling, 
Marcinkiewicz-Zygmund Inequalities, Approximation,  and
Quadrature Rules}
\author{Karlheinz Gr\"ochenig}
\address{Faculty of Mathematics \\
University of Vienna \\
Oskar-Morgenstern-Platz 1 \\
A-1090 Vienna, Austria}
\email{karlheinz.groechenig@univie.ac.at}
\subjclass[2010]{}
\date{}
\thanks{K.\ G.\ was
  supported in part by the  project P31887-N32  of the
Austrian Science Fund (FWF)}

\maketitle

\section{Introduction}

This article is motivated by two perennial questions of approximation
theory: Assume that a finite number of  samples of a continuous function $f$ on some
compact set is given,  (i) find a good or optimal approximation of $f$
from these samples and  derive  error estimates,  and   (ii) approximate an
integral $\int f$ from these samples and derive error estimates, in
other words, find a quadrature rule based on the given samples.

 We argue, completely  in line with the tradition of approximation theory, that these
questions are best answered by means of \mz\ families and  inequalities. 
Roughly speaking, a \mz\ family is a double-indexed set of points
$\xkn $ such that the sampled $\ell ^2$-norm of the $n$-th layer $\sum
_{k} |p(\xkn )|^2 $ is an equivalent norm for the space of ``polynomials''
of degree $n$ with uniform constants.
Our main result then  shows that the existence of a \mz\ family already
implies (i) approximation theorems from pointwise samples of a function, and (ii)
quadrature rules. This is, of course, folkore, and the content of an
abundance of results in approximation theory and numerical analysis 
on many  levels of generality. In the literature \mz\ families are  constructed for
the purpose of quadrature rules and approximation theorems~\cite{FM10,MNW01},
our main insight is that  quadrature rules
and approximation theorems follow automatically  from  a
\mz\ family.   It is one of our objectives to explain
this conceptual  hierarchy:  \mz\ families are first, then  quadrature rules and
approximation theorems come for free. 


The main novelty of our contribution is a completely elementary
derivation of approximation theorems and quadrature rules based on the 
existence of a \mz\ family. This derivation is fairly simple and is  based solely on the
basic definitions of \mz\ families, orthogonal projections, Sobolev
spaces, and least square problems. The assumptions are minimal and
only require an orthonormal basis $\{\phi _k\}$  and an associated
non-decreasing sequence of ``eigenvalues'' $\{\lambda _k\} \subseteq
\bR ^+$. This set-up is similar to the one in~\cite{FM10,FM11,MM08}.  

Our point of view is informed by the theory of non-uniform
sampling of bandlimited functions and their discrete analogs developed
in the 1990's by many groups~\cite{BH90,FGtp94,fgs95,PST01,Sunw02}.
Indeed, a sampling theorem is simply a \mz\ inequality (upper and
lower) for a fixed function space, and some of the first \mz\
inequalities for scattered points on the torus (or non-uniform sampling points)
were derived in this context~\cite{Gro93c}. The  method  in this paper 
was essentially developed in~\cite{Gro99} for the local approximation of
bandlimited functions by trigonometric polynomials from samples. 

Several technical aspects deserve special mention.

(i) In general, the error estimates for quadrature rules depend on a
covering radius (or mesh size) or on the number of nodes that  arise in a particular construction of a
\mz\ family~\cite{BCCG14,EGO17}. In our derivation the constants depend only on the
condition number of the \mz\ family and thus on their definition.

(ii)   Whereas quadrature
rules are often connected to \mz\ inequalities with respect to the
$L^1$-norm~\cite{FM10}, we derive such rules from \mz\ inequalities in
the $L^2$-norm by means of  frame theory.  The frame approach to
quadrature rules  is 
motivated by a  question of N.~ Trefethen about convergence of
the standard  quadrature rules  after a perturbation of the uniform
grid, see~\cite{TW14,AT17}.

(iii)  In several articles on \mz\ families \cite{FM10,MM08,OP12} the polynomial growth of
the spectral function $\sum _{k=1}^n |\phi _k(x)|^2$ is used
implicitly  or as a hypothesis.  In our treatment
(Lemma~\ref{critval}),  the growth of the spectral function  is  
related to the critical Sobolev exponent and leads to  explicit and 
transparent error estimates.  In hindsight the appearance
of the spectral function is not surprising, as it is the reciprocal of
the Christoffel function associated to an orthonormal basis (or to a
set of orthogonal polynomials), and is thus absolutely fundamental for
polynomial interpolation and quadrature rules. See~\cite{Nev86} for an extended survey. 

\vspace{ 2mm} 
The paper is organized as follows: the end of this introduction
provides a brief survey of related literature. In Section~2 we 
introduce \mz\ families for the torus and prove the resulting
approximation theorems and quadrature rules. In Section~3 we treat the
same question in more  generality on a compact space with a given
orthonormal basis and corresponding eigenvalues. With  the appropriate
definitions of Sobolev spaces and error terms, the formulation of the
main results and the proofs are then identical. In a sense Section~2
is redundant, but we preferred to separate the proofs from the
conceptual work. This separation reveals the simplicity of the
arguments more clearly. 


%

\subsection{Discussion of related work}
It is impossible to  do full justice to the extensive literature on \mz\ inequalities,
 approximation theorems, and quadrature rules,  we will therefore
 mention only some aspects and apologize for any  omissions. 

The classical theory of \mz\ inequalities deals with the interpolation
by polynomials and the associated quadrature rules and is surveyed
beautifully in~\cite{Lub98}. An  extended recent  survey with a
complementary point of view
is contained in~\cite{DTT19}. 

(i) \emph{\mz\ families on the torus and trigonometric polynomials.} 
\mz\ inequalities for scattered nodes were considered in the theory of
nonuniform sampling~\cite{FGtp94}. An early example of  \mz\
inequalities  on the torus  is contained in the
estimates of~\cite[Thm.~4]{Gro93c}. \mz\ with respect to different measures
measures were then studied in~\cite{Erd99,Lub99,MT00,RS97}.
A complete   characterization of
 \mz\ families on the torus with respect to Lebesgue measure  in terms of suitable
 Beurling densities was given by Ortega-Cerd\`a
 and Saludes~\cite{OS07}. 

 (ii) The next phase concerned \emph{\mz\ inequalities and quadrature
   on the sphere}: The goal of \cite{MNW01} is the construction of
 good quadrature rules on the sphere via \mz\ inequalities, sufficient
 conditions for \mz\ families are obtained in~\cite{MP14}. 
  Necessary  density conditions for \mz\ inequalities on the
 sphere are derived in~\cite{Mar07}, and \cite{MOC10} studies the
 connection between \mz\ families and Fekete points on the sphere.
 \cite{HS06} derives worst case errors for quadratures on the  sphere. 

 (iii) \emph{\mz\ families on metric measure spaces.} The most
 general constructions of \mz\ families and quadrature rules  are due to Filbir and Mhaskar
 in a series of papers~\cite{FM10,FM11,MM08,Mh18}. They work  for metric measure spaces with
 Gaussian estimates for the heat kernel associated to an orthonormal
 basis. This theory includes in particular \mz\ families on compact
 Riemannian manifolds. A related theory is contained in~\cite{FFP16}
 whose goal is the construction of frames for Besov spaces. Again,
 necessary density conditions for \mz\ families on compact Riemannian
 manifolds have been derived by Ortega-Cerd\`a and
 Pridhnani~\cite{OP12}. 

(iv) \emph{Approximation of functions from samples via least squares:}
In general the approximating polynomials do not interpolate, therefore
the best approximation of a given function by a ``polynomial'' is
obtained by solving a least squares problem. A \mz\ inequality then
implies bounds on the condition number of the underlying matrix. This
connection appears, among others,  in~\cite{Gro99,FM11} and is
highlighted in Proposition~\ref{prop1}.  Modern
versions use random sampling to generate \mz\ inequalities. This
aspect was studied in~\cite{BG04,SZ04} for random sampling in
finite-dimensional subspaces, \cite{CM17,ABC19,AC19} contain recent studies of the stability of
least squares reconstruction. We point out in particular~\cite{CM17}
where the Christoffel function is identified as the optimal weight for
a given probability measure. Finally we highlight the series of papers
on generalized sampling~\cite{AH12,AHP13} as an alternative approach to the
approximation of functions from finitely many linear measurements. In this
case the constants in the error estimates are formulated with the angle between
subspaces rather than with the condition number of the \mz\ family.

\section{Approximation of functions on the torus from nonuniform samples}
As a  model example where the technique is completely transparent, we
first deal with nonuniform sampling on the torus $\bT $. On $\bT$ the
approximation spaces are  the space of trigonometric polynomials $\cT
_n$ of degree $n$, i.e.,
 $p\in \cT _n$, if $p(x) = \sum _{k=-n}^n c_k e^{2\pi i k x}$. 

 \begin{definition} \label{defmz}
   Let  $\cX    = \{ \xkn : n\in \bN , k=1, \dots , L_n \}$ be  a
   doubly-indexed set  of points in $\bT \simeq (-1/2,1/2] $ and $\tau
   =  \{ \tkn : n\in \bN , k= 1, \dots , L_n \}\subseteq (0,\infty ) $ be  a
   family of non-negative weights. Then $\cX $ 
   is called a \mz\ family, if there
  exist constants $A,B >0$ such that 
  \begin{equation}
    \label{eq:1}
A \|p\|_2^2 \leq \sum _{k=1} ^{L_n} |p(\xkn )|^2 \tau _{n,k}    \leq B
\|p \|_2^2 \qquad \text{ \emph{for all} } p \in \cT _n \, .
  \end{equation}
The ratio $\kappa = B/A$ is the global condition number of the \mz\
family, and $\cX _n =\{ \xkn : k=1, \dots , L_n \}$ is the $n$-th
layer of $\cX $. 
\end{definition}

The point of  Definition~\ref{defmz} is that the constants are uniform in the
degree $n$ and that usually  $\cX _n$  contains more than
$\mathrm{dim}\, \cT _n = 2n+1$ points,  so
that it is not an interpolating set for $\cT _n$. Weights are
omnipresent in the classical theory of \mz\ inequalities~\cite{Lub99},
  in sampling
theory they are used to improve  condition numbers~\cite{fgs95}, and
in Fourier sampling they serve as density compensating factors. 
Currently  they  play an  important role in weighted least squares
problem in statistical estimation in~\cite{CM17,ABC19,AC19}.

Given the samples $\{ f(\xkn ) \}$ of a continuous  function $f $ on
$\bT $ on the $n$-th layer $\cX _n$, we first need to approximate $f$
using only these samples. For this we  solve a sequence of least
squares problems with samples taken from the $n$-th layer $\cX _n$: 
\begin{equation}
  \label{eq:8}
p_n = \mathrm{argmin} _{p\in \cT _n} \sum _{k=1} ^{L_n} |f(\xkn ) -
p(\xkn )|^2 \tkn \, .
\end{equation}
This procedure yields a  sequence of trigonometric
polynomials for every $f\in C(\bT )$. In general,  these polynomials
do not interpolate the given $f$ on $\cX _n$, 
but they yield the best $\ell ^2$-approximation of the data $\{f(\xkn )\}$ by a
trigonometric polynomial in $\cT _n$. Therefore  $p_n$ is usally called a
quasi-interpolant.

The question is now  how the $p_n$'s approximate  $f$ on all of $\bT
$. As always in approximation theory, the answer depends on the
smoothness of $f$. For this we use the standard Sobolev spaces $H^\sigma
(\bT )$ with  norm 
\begin{equation}
  \label{eq:2}
  \|f\| _{H^\sigma } = \Big(\sum _{k\in \bZ } |\fhat (k)|^2
  (1+k^2)^{\sigma } \Big)^{1/2} \, ,
\end{equation}
where  $\fhat (k) = \int _0^1 f(x)
e^{-2\pi i kx} \, dx$ is the $k$-th Fourier coefficient of $f$. 

Our main theorem asserts the  convergence of the
quasi-interpolants of $f$. 

\begin{tm} \label{tm1}
Let $\cX $ be a \mz\ family with associated weights $\tau $ and
condition number $\kappa = B/A$. 

(i)   If $f\in H^\sigma $ for $\sigma >1/2$, then 
  \begin{equation}
    \label{eq:3}
    \|f-p_n\|_2  \leq  C_\sigma \sqrt{1+\kappa ^2} \|f\|_{\hs }
    n^{-\sigma +1/2} \, ,
  \end{equation}
with a constant depending on $\sigma$ (roughly $C_\sigma \approx
(\sigma -1/2)^{-1/2}$).
 
(ii) If $f$ extends to an analytic function on an strip  $\{z\in \bC
: |\mathrm{Im} z| <  \rho  _0 \}$, then the convergence is geometric, i.e.,
  \begin{equation}
    \label{eq:3a}
    \|f-p_n\|_2 = \cO (e^{-\rho n}) \, .
  \end{equation}
  for every $\rho <\rho _0$.
\end{tm}



The proof starts with   the orthogonal projection  $P_nf(x) = \sum _{|k|\leq n}
\fhat (k) e^{2\pi i kx} $  of $f$ onto the trigonometric 
polynomials $\cT _n$. Note that $P_nf $ is the $n$-th partial sum of
the Fourier series of $f$. The proof of Theorem~\ref{tm1} is based on the orthogonal decomposition
\begin{equation}
  \label{eq:4}
  \|f- p_n \|_2^2 =   \|f- P_nf \|_2^2 +  \|P_nf - p_n \|_2^2 \, .
\end{equation}
The first term measures how fast the partial sums of the Fourier
series converge to $f$, whereas the second, and more interesting, term
compares the best $L^2$-approximation of $f$ in $\cT _n$ with the
approximation $p_n$ obtained from the samples of $f$ on $\cX _n$. 

\subsection{Sampling and embeddings in $\hs $ }

Before entering the details of the proof, we state some well-known 
facts about the Sobolev space $\hs $. 
\begin{lemma} \label{lm1}
Assume that $\sigma > 1/2$. 

(i) Sobolev embedding: Then $\hs (\bT )  $ is continuously embedded in
$C(\bT )$. 
  
(ii) Convergence rate: for all $f\in \hs (\bT ) $
\begin{equation}
  \label{eq:5}
  \|f-P_nf \|_\infty  \leq  \|f \|_{\hs } \, \phi _\sigma (n) \, ,
\end{equation}
where $\phi _\sigma (n) = (\sigma -1/2)^{-1/2}\, n^{-\sigma
  +1/2}$.

(iii) Sampling in $\hs $: If $\cX $ satisfies the sampling
inequalities \eqref{eq:1} and  $f\in \hs $, then 
\begin{equation}
  \label{eq:6}
\sum _{k=1}^{L_n} |f(\xkn )|^2 \tkn \leq  B \|f\|_\infty ^2  \leq B C_\sigma ^2 \|f\|_{\hs }^2 \, , 
\end{equation}
 \end{lemma}

 \begin{proof}
(i) and (ii) are standard (and also  follow from  Lemma~\ref{sobgen}).

(iii)  The sampling inequalities
\eqref{eq:1} applied to the constant function $p\equiv 1$  with
$\|p\|_2=1$ yield
\begin{equation}
  \label{eq:7}
  A \leq \sum _{k=1}^{L_n} \tkn \leq B \, .
\end{equation}
The claim follows from \eqref{eq:7} and the Sobolev embedding. 
 \end{proof}


\subsection{Quasi-interpolation versus projection}
To estimate the norm $\| P_n f - p_n\|_2$, let us introduce the
vectors and matrices that arise in the explicit solution of the least squares
problem \eqref{eq:8}.
Let $$
y_n = (\tau _{n,1}^{1/2} f(x_{n,1}), \dots , \tau _{n,L_n}^{1/2}
f(x_{n,L_n})) \in \bC ^{L_n}$$
be the
given data vector, and $U_n $ be the $L_n \times (2n+1)$-matrix (a
Vandermonde matrix) with
entries 
\begin{equation}
  \label{eq:10}
(U_n)_{kl} =  \tkn ^{1/2} e^{2\pi i x_{n,k} l} \qquad k=1,\dots , L_n, |l| \leq n
\, .  
\end{equation}
We write 
\begin{equation}
  \label{eq:11}
  T_n = U_n ^* U_n \, . 
\end{equation}
For the numerical construction of $p_n$ we  note that $T_n$ is a
Toeplitz matrix and thus accessible to fast algorithms~\cite{Gro93c,fgs95,PST01}. For our
analysis we collect the following facts. 

\begin{lemma} \label{lem2a}
  Assume that $\cX $ is a \mz\ family. Then

(i) the spectrum of $T_n$ is contained in the
interval $[A,B]$ for all $n\in \bN $, and 

(ii)  the  solution of the least
squares problem \eqref{eq:8} yields a trigonometric polynomial $p_n =
\sum _{|k| \leq n} a_{n,k} e^{2\pi i kx}
\in \cT _n$ with a coefficient vector $a_n \in \bC ^{2n+1}$ given by
\begin{equation}
  \label{eq:12}
  a_n = T_n \inv U_n ^* y_n \, .
\end{equation}
\end{lemma}

\begin{proof}
  (i) Note that for $p\in \cT _n$  and $p(x) = \sum _{|l|\leq n} a_l
  e^{2\pi i l x}$ the point evaluation at $\xkn \in \cX _n$ is
  precisely 
$$
 \tkn ^{1/2} p(\xkn ) = (U_n a )_k    \, ,
$$ and  the sampled $2$-norm is 
$$
\sum _{k=1} ^{L_n} |p(\xkn )|^2 \tau _{n,k}  = \langle U_n a , U_n
a\rangle = \langle T_n a, a \rangle  \, .
$$
By \eqref{eq:1} the spectrum of $T_n$ is contained in the interval
$[A,B]$. 

(ii) This is the standard formula for the solution of a least squares
problem by means of the Moore-Penrose pseudo-inverse $U_n ^\dagger =
(U_n^* U_n)\inv U_n ^* = T_n\inv U_n^*$. 
\end{proof}

Here is  the decisive estimate for the  second term in
\eqref{eq:4}. The following lemma relates the solution to the least
squares problem~\eqref{eq:8} to the best approximation of $f$ in $\cT
_n$. Compare~\cite{Gro99,Gro01b} for an early use of this argument.  

\begin{prop}\label{prop1}
Let $p_n$ be the solution of the least squares problem
\eqref{eq:8}. Then 
\begin{equation}
  \label{eq:13}
  \|P_nf - p_n \|_2^2 \leq A^{-2} B \sum _{k=1}^{L_n} |f(\xkn ) - P_n
  f(\xkn )|^2 \tkn \, .
\end{equation}
  \end{prop}

  \begin{proof}
 Let $f_n \in \bC ^{2n+1}$ be the Fourier coefficients of the
projection $P_nf$, i.e, $f_n = (\fhat (-n), \fhat (-n+1), \dots, \fhat
(n-1), \fhat (n))$. Then by Plancherel's theorem
$$
\|P_n f - p_n \|^2_2 = \| f_n  - a_n \|^2_2 \, ,
$$
where the norm on the left-hand side is taken in $L^2(\bT)$ and on the
right-hand side in $\bC ^{2n+1}$. Using \eqref{eq:12} for the solution   of
the least squares problem \eqref{eq:8} and $T_n = U_n^* U_n$, we  obtain
\begin{align*}
   \| f_n  - a_n \|^2_2 &= \|f_n - T_n \inv U_N^* y_n\|_2^2 \\
&= \| T_n \inv U_n^* (U_n  f_n - y_n)\|_2^2 \\
& \leq A^{-2} B \|U_n f_n - y_n\|_2^2 \, , 
\end{align*}
      because the operator norm of $T_n\inv $ is bounded by $A\inv $
      and the norm of $U^*_n$ is bounded by $\|U_n ^*\| = \|U_n\|=
      \|U_n^* U_n \|^{1/2} = \|T_n \|^{1/2} \leq B^{1/2}$. 
Finally, 
$$
(U_nf_n)_k = \tkn ^{1/2} \sum _{|l| \leq n} e^{2\pi i 
  \xkn l}  \hat{f}(l) = \tkn ^{1/2} P_n
f(\xkn )  \, ,
$$
and thus
$$
\|U_n f_n - y\|_2^2 = \sum _{k=1} ^{L_n} |P_nf(\xkn ) - f(\xkn )|^2
\tkn  \, ,
$$ 
and the statement is proved. 
  \end{proof}

\subsection{Proof of Theorem~\ref{tm1}}
\begin{proof}
(i) We use  the orthogonal  decomposition  $
\|f- p_n \|_2^2 =   \|f- P_nf \|_2^2 +  \|P_nf - p_n \|_2^2 \, .
$
Then  Lemma \ref{lm1}(ii) yields 
$$
  \|f- P_nf \|_2^2 \leq \|f- P_nf \|_\infty^2 \leq  \,\|f\|_{\hs } ^2 \,
\phi _\sigma (n)^2
  \, ,
$$
and Proposition~\ref{prop1} yields 
$$
  \|P_nf - p_n \|_2^2 \leq A^{-2} B \sum _{k=1}^{L_n} |f(\xkn ) - P_n
  f(\xkn )|^2 \tkn \, .
$$
We now apply \eqref{eq:7} and Lemma~\ref{lm1}(ii) to $f-P_nf \in \hs
(\bT )$ and
continue the inequality as 
$$
\sum _{k=1}^{L_n} |f(\xkn ) - P_n
  f(\xkn )|^2 \tkn  \leq B \|f-P_nf\|_{\infty} ^2 \leq B \|f\|_{\hs }
  ^2 \phi _\sigma (n)^2 \, .
$$
The combination of these inequalities yields the final error estimate
\begin{equation}
  \label{eq:14}
  \|f-p_n\|_2^2 \leq     \big( 1 +
  \frac{B^2}{A^2} \big) \|f\|_{\hs } ^2 \, \phi _\sigma (n) ^2 \, .
\end{equation}
Since $\phi _\sigma (n)  = \cO (n^{-\sigma +1/2})$, Theorem~\ref{tm1} is
proved. 

(ii) If $f$ can be extended to an analytic function of the strip $\{z\in \bC
: |\mathrm{Im} z| <  \rho  _0 \}$, then its Fourier coefficients decay
exponentially as $|\fhat (k) | \leq c_\rho e^{-\rho |k|}$ for every
$\rho <\rho _0$ with an appropriate constant. Consequently
\begin{equation}
  \label{eq:abc}
\|f-P_nf\|_\infty \leq \sum _{|k|>n} |\fhat (k)| \leq  c_\rho \sum
_{|k|>n} e^{-\rho |k|} = \frac{2c_\rho}{e^{\rho}-1} e^{-\rho n} \, ,
\end{equation}
which proves the exponential decay. 
\end{proof}

\subsection{Quadrature Rules}
To deduce a set of quadrature rules, we use frame theory to obtain
suitable weights. See~\cite{Chr16,DS52} for the basic facts. The
following argument is typical in sampling theory, whereas often the
derivation of quadrature rules relies on abstract functional analytic
arguments as in~\cite{MNW01}.

Let $k^{(n)}_x\in \cT_n$ be the reproducing kernel of $\cT _n$ defined by
$p(x) = \langle p, k^{(n)}_x \rangle $ for all $\pi \in \cT _n$ and $x\in
\bT $. In fact, $k^{(n)}_x(y) = \frac{\sin (2n+1)\pi (y-x)}{\sin \pi (y-x)}$ is just the Dirichlet kernel for $\cT
_n$. In the language of frame theory each inequality of \eqref{eq:1}
simply says that $\{\tkn ^{1/2} \, k^{(n)}_{\xkn
} : k=1, \dots , L_n\}$ is a frame for $\cT _n$ with  frame
bounds $A,B>0$ independent of $n$. Equivalently, the   associated frame
operator
$S_np = \sum _{k=1}^{L_n} \tkn  \langle p, k^{(n)}_{\xkn } \rangle k^{(n)}_{\xkn }
$
is invertible on $\cT _n$ for every $n\in \bN $ and we obtain the dual
frame $e_{n,k} =  S_n \inv ( \tkn ^{1/2}  k^{(n)}_{\xkn }) $. The factorization
$S_n\inv S_n = \mathrm{I}_{\cT _n}$ yields the following
reconstruction formula for all  trigonometric polynomials $p \in \cT
_n$ from their samples:    
\begin{equation}
  \label{eq:17}
  p = \sum _{k=1}^{L_n} \tkn   \langle p , k^{(n)}_{\xkn } \rangle
  S\inv k^{(n)}_{\xkn } = \sum _{k=1}^{L_n} \tkn ^{1/2} p(\xkn ) \,
  e_{n,k} \,  
\end{equation}
Furthermore,  $\{e_{n,k} : k=1, \dots, L_n\}$ is a frame for $\cT _n$
with frame bounds $B\inv $ and $ A\inv $, again independent of $n$.
This property implies in particular that for the constant 
function $1$ with $\|1\|_2 = 1$ we have
\begin{equation}
  \label{eq:a1}
  \sum _{k=1} ^{L_n} |\langle  1, e_{n,k}  \rangle |^2 \leq A\inv \|1\|_2
  = A\inv \, .
\end{equation}
We now define the weights for the quadrature rules by 
\begin{equation}
  \label{eq:18}
  w_{n,k} = \tkn ^{1/2} \langle  e_{n,k} ,1 \rangle =  \tkn ^{1/2} \int
  _{-1/2} ^{1/2} e_{n,k}(x) \, dx \, , 
\end{equation}
and the corresponding quadrature rule by 
\begin{equation} \label{eq:18b}
  I_n(f) = \sum _{k=1}^{L_n} f(\xkn ) w_{n,k} \, .
\end{equation}
We also  write $I(f) = \int _{-1/2} ^{1/2} f(x) \, dx $ for the
integral of $f$ on $\bT $. 

As a consequence of the definitions we obtain the following easy
properties of this quadrature rule.
\begin{lemma} \label{lem-quad}
  Let  $w_{n,k}$ and $I_n$ be  defined as in \eqref{eq:18} and
  \eqref{eq:18b}.

  (i) Then the quadrature rule $I_n$ is exact on $\cT _n$, i.e.,
  $I_n(p) = I(p)$ for all $p \in \cT _n$.

  (ii) For $f\in \hs (\bT )$  we have
  \begin{equation}
    \label{eq:a2}
    |I_n(f)|^2 \leq A\inv \sum _{k=1}^{L_n} |f(\xkn )|^2 \tkn  \leq
    \tfrac{B}{A}\|f\|_\infty ^2 \, .
  \end{equation}
\end{lemma}

\begin{proof}
  (i) follows from \eqref{eq:17}. For (ii) we use
  \begin{equation*}
    |I_n(f)|^2 \leq \big( \sum _{k=1}^{L_n} |f(\xkn )|^2 \tkn \Big) \,
\Big(  \sum _{k=1}^{L_n} |\langle  e_{n,k} , 1 \rangle )|^2  \Big)
\leq \frac{B}{A} \|f\|_\infty ^2 \, .
  \end{equation*}
\end{proof}

As a consequence of Theorem~\ref{tm1} we obtain the following
convergence theorem. 

\begin{tm}
  \label{tm2}
Let $\cX $ be a \mz\ family with weights $\tau $ and let $\{I_n: n \in \bN \}$ be the
associated sequence of  quadrature rules. 

(i) If $f\in C(\bT )$, then 
\begin{equation}
  \label{eq:190}
  |I(f)  - I_n (f)| \leq (1+\sqrt{\kappa} )  \inf _{p \in \cT _n}
  \|f-p\|_\infty \, .  
\end{equation}
Consequently, if $f\in C^\sigma (\bT )$, then $  |I(f)  - I_n (f)| =
\cO (n^{-\sigma })$.

(ii) If $f\in \hs $ for $\sigma >1/2$, then 
\begin{equation}
  \label{eq:19}
  |I(f)  - I_n (f)| \leq (1+ \sqrt{\kappa })   \|f\|_{\hs } \phi
  _\sigma (n) \, ,
\end{equation}
with $\phi _\sigma (n) = (\sigma -1/2)^{-1/2} n^{-\sigma +1/2}$.

(iii) If $f$ extends to an analytic function on a strip  $\{z\in \bC
: |\mathrm{Im}\, z| <  \rho _0 \}$, then  for $\rho <\rho _0$
$$
|I(f) - I_n (f)| = \cO (e^{-\rho n}) \, .
$$
\end{tm}

\begin{proof}
(i) and (ii)  Let  $P_nf$ is the orthogonal projection of $f$
onto $\cT _N$ and  $q_n$ be the best approximation of $f$ in $\cT _n$  with respect
to $\| \cdot \|_\infty $. 
  Since   $I_n$ is exact on
  $\cT _n$, we  have $I(P_nf)  = I_n (P_nf)$ and $I(q_n) =
  I_n(q_n)$. Then we obtain  with \eqref{eq:a2} that 
  \begin{align*}
    |I(f) - I_n(f)| &\leq |I(f-q_n)| + |I_n(q_n - f)| \\
                    &\leq \|f-q_n\|_\infty +   (B/A)^{1/2} \|f-q_n\|_\infty \\
                    & =  (1+\sqrt{\kappa } ) \inf _{p\in \cT _n}
                      \|f-p\|_\infty \, .
  \end{align*}
For the approximating polynomial $P_nf$ we obtain 
  \begin{align*}
    |I(f) - I_n(f)| &\leq |I(f-P_nf)| + |I_n(P_nf - f)| \\
                    &\leq \|f-P_nf\|_\infty +   (B/A)^{1/2} \|f-P_nf\|_\infty \\
    &\leq (1+\sqrt{\kappa } ) \,    \|f\|_{\hs } \, \phi _\sigma (n)
      \, ,   \end{align*}
    with $\phi _\sigma (n) = (\sigma -1/2)^{-1/2} n^{-\sigma +1/2}$
  by Lemma~\ref{lm1}.
  
(iii) If $f$ can be extended to the strip $\{z\in \bC
: |\mathrm{Im} z| <  \rho  _0 \}$, then we use the error estimate
\eqref{eq:abc} for $\|f-P_nf\|_\infty$ and obtain 
$$
  |I(f) - I_n(f)| \leq (1+ (B/A)^{1/2}) \|f-P_nf\|_\infty \leq c_\rho '
  (1+\sqrt{\kappa }) \, e^{-\rho n}
$$
for every $\rho <\rho _0$ with a constant $c_\rho '$ depending on $\rho
$. 
\end{proof}

Theorem~\ref{tm2} answers a question of N.\ Trefethen~\cite{AT17}
about the rate of convergence of the standard quadrature rules for
scattered nodes.

\section{General Approximation Theorems}
 Theorem~\ref{tm1} and the quadrature rule of Theorem~\ref{tm2}
required hardly any tools, and the proofs  use  only the definitions of  \mz\
inequalities, Sobolev spaces, and the solution formula for least square
problems. We will now show that the results for $\bT $  can be
extended   significantly  with a mere change of
notation.

For an axiomatic approach to approximation theorems from samples and
quadrature rules, we assume that $M$ is a 
compact space and $\mu $ is a probability measure on $M$. Furthermore, 

(i)  $\{\phi _k : k\in \bN\}$ is an orthonormal basis for $L^2(M,\mu
)$.  In agreement with the notation for Fourier series, we write $\fhat (k)
= \langle f, \phi _k\rangle =
\int _M f(x) \overline{\phi _k(x)} \, d\mu (x)$ for the $k$-th
coefficient, so that the orthogonal expansion is $f= \sum _k \fhat (k)  \phi _k$. 

(ii) Next, let  $\lambda _k \geq 0$ be  a non-decreasing  sequence  with
$\lim _{k\to \infty } \lambda _k = \infty $. 
The associated Sobolev
space $\hs (M)$ is defined by
\begin{equation}
  \label{eq:20}
  \hs (M) = \{ f\in L^2(M): \|f\|_{\hs } ^2 = \sum _{k=1}^\infty
  |\fhat (k)|^2 (1+\lambda _k ^2)^{\sigma } < \infty \} \, . 
\end{equation}
With $\cP _n$ we denote the set of ``polynomials'' of degree $n$ on
$M$ by 
\begin{equation}
  \label{eq:21}
  \cP _n = \{ p \in L^2(M): p = \sum _{k: \lambda _k \leq n} \fhat (k)
  \, \phi _k \} \, .
\end{equation}
 This space is  finite-dimensional because $\lim _{k\to \infty }
 \lambda _k = \infty $.  The definition of $\cP _n$ encapsulates the
 appropriate  notion of bandlimitedness with respect to the basis
 $\{ \phi _n \}$; in the  terminology of \cite{FM10}  the functions 
 in $\cP_n $ are called  diffusion polynomials.

 Note that both $\hs (M)$ and $\cP _n$ depend on the orthonormal basis
 and on the sequence $\{\lambda _k\}$. We may think of $\{\phi _k\}$
 as the set of eigenfunctions of an unbounded positive operator on
 $L^2(M,\mu )$ with eigenvalues $ \lambda _k$.  

 The following table illustrates the transition from the set-up in
Section~2 to the general theory.


\begin{table}[h!]
  \begin{center}
    \caption{Generalization}
    \label{tab:table1}
    \begin{tabular}{l|l} 
      \emph{from torus $\bT$} & \emph{to manifold $M$} \\
      \hline
      ONB $e^{2\pi i k x}$ for $L^2(\bT )$ &   ONB $\phi _k (x) $ for  $L^2(M,\mu )$ \\
    Fourier coefficients $\fhat (k) $ & Fourier coefficients~$ \langle f, \phi _k \rangle $ \\
    Eigenvalues $ k^2$ of $-\frac{1}{4\pi ^2}\tfrac{d^2}{dx^2} $ &  ``Eigenvalues'' $\lambda _k \geq 0 $, $\lambda _k \to \infty $ \\
    Sobolev space $\hs (\bT)$ &   Sobolev space $\hs (M) $  \\
    Trigonometric polynomials $\cT _n$  &   ``Polynomials'' $\cP _n
   $, $p = \sum _{k: \lambda _k \leq n } \fhat (k) \, \phi _k  $
     \end{tabular}
  \end{center}
\end{table}

 For meaningful statements we make the following natural assumptions:

 (i)  Every basis element $\phi _k$ is continuous (and thus bounded)
 on $M$ and $\phi
_1 \equiv 1$. Then the point evaluation $p \to p(x)$ makes sense on
$\cP _n$ and $\int _M f d\mu = \langle f, \phi _1\rangle$. 

(ii) There exists  a critical index
$\sigma _{\mathrm{crit}}$  such that for $\sigma > \scrit $ the sum
\begin{equation}
  \label{eq:a3}
C_\sigma ^2 = \sup _{x\in M}  \sum _{k=1}^\infty |\phi _k(x)|^2
(1+\lambda _k ^2)^{-\sigma } <\infty
  \, 
\end{equation}
converges\footnote{In  \cite{BCCG14} the expression  $\sum _{k=1}^\infty \phi _k(x) \phi _k(y)
(1+\lambda _k ^2)^{-\sigma }$ is called the Bessel kernel associated
to the orthonormal basis $\{\phi _k\}$.}.

(iii) The error estimates will be in terms of the remainder function 
\begin{equation}
  \label{eq:26}
  \phi _\sigma  (n) = \sup _{x\in M} \Big(\sum _{k: \lambda _k > n} |\phi _k(x)|^2 (1+ \lambda _k ^2)^{-\sigma
} \Big)^{1/2}  \, .
\end{equation}
By \eqref{eq:a3}  $\phi _\sigma (n) \to 0$ as $n\to \infty $.  

A \mz\ family  $\cX $ for $M$  is  a doubly-indexed set $\cX = \{ \xkn
: n\in \bN , k=1 , \dots , L_n\} \subseteq
M$ with associated weights $\{ \tkn \}$, such that 
\begin{equation}
  \label{eq:1a}
A \|p\|_2^2 \leq \sum _{k=1} ^{L_n} |p(\xkn )|^2 \tau _{n,k} \leq B
\|p \|_2^2 \qquad \text{ for all } p \in \cP _n \, ,
\end{equation}
with constants $A,B >0$ independent of $n$.

\subsection{Embeddings}
We first prove the versions for approximation and embedding in the
general context.

\begin{lemma}
  \label{sobgen}
  (i) If $\sigma > \scrit $, then $\hs (M) \subseteq C(M)$, in fact,
  $\| f \|_\infty \leq C_\sigma \, \|f\|_{\hs }$.  

  (ii) 
For $f\in C(M)$ 
  \begin{equation}
    \label{eq:a4}
    \|f-P_nf \|_2 \leq \|f-P_nf\|_\infty \leq \|f\|_{\hs } \phi
    _\sigma (n) \, .
  \end{equation}

(iii) Assume that $\cX $ is a \mz\ family with weights $\tau $.  For $f\in \hs (M) $ we have 
\begin{equation}
  \label{eq:6a}
  \sum _{k=1}^{L_n} |f(\xkn )|^2 \tkn \leq B C_\sigma ^2 \|f\|_{\hs }^2 \, .
\end{equation}
\end{lemma}

\begin{proof}
(i) and (ii): 
Let $f = \sum _{k=1}^\infty \fhat (k) \phi _k$. Then with the
Cauchy-Schwarz inequality, 
\begin{align*}
|f(x) | \leq \Big(\sum _{k=1}^\infty |\fhat (k)|^2 (1+\lambda
          _k ^2)^\sigma   \Big)^{1/2} \, \sup _{x\in M} \Big(\sum
          _{k=1}^\infty |\phi _k(x)|^2 (1+\lambda
          _k ^2)^{-\sigma }  \Big)^{1/2}  =  C_\sigma \|f\|_{\hs } \, .
\end{align*}
The approximation error is estimated by 
\begin{align}
    \|f-P_nf\|_2 & \leq   \|f - P_n f \|_\infty \notag \\
  &\leq \|f \|_{\hs }  \sup _{x\in M}  \Big(\sum _{k: \lambda _k > n} |\phi _k(x)|^2 (1+ \lambda _k )^{-\sigma
} \Big)^{1/2} = \|f\|_{\hs } \, \phi _\sigma (n) \, .  \label{eq:25}
\end{align}
In the first inequality we have used the fact that $\mu $ is a
probability measure. 

(iii) is proved  as in Lemma~\ref{lm1}.  The sampling inequalities
\eqref{eq:1} applied to the constant function $p\equiv 1$  with
$\|p\|_2=1$ yields
\begin{equation}
  \label{eq:7a}
  A \leq \sum _{k=1}^{L_n} \tkn \leq B \, .
\end{equation}
Consequently, with the embedding of (i), we have 
$$
\sum _{k=1}^{L_n} |f(\xkn )|^2 \tkn \leq  \|f\|_\infty ^2 \sum
_{k=1}^{L_n}  \tkn \leq B C_\sigma ^2 \|f\|_{\hs }^2 \, .
$$ 
\end{proof}

\subsection{Quasi-interpolation versus projection}

 To produce optimal
approximations of $f$ in $\cP _n$ from the samples $\cX _n= \{ \xkn : k=1 , \dots
, L_n\} $ in the $n$-th layer of $\cX $, we solve the  sequence of least
squares problems
\begin{equation}
  \label{eq:8b}
p_n = \mathrm{argmin} _{p\in \cT _n} \sum _{k=1} ^{L_n} |f(\xkn ) -
p(\xkn )|^2 \tkn \, .
\end{equation}
Now let $y_n = (\tau _{n,1} ^{1/2} f(x_{n,1}), \dots , \tau
_{n,L_n}^{1/2} f(x_{n,L_n})) \in \bC ^{L_n}$ be the
given data vector, and $U_n $ be the $L_n \times \mathrm{dim}\, \cP _n$-matrix  with
entries 
\begin{equation}
  \label{eq:10a}
(U_n)_{kl} = \tkn ^{1/2} \phi _l (x_{n,k})  \qquad k=1,\dots , L_n,
 l = 1, \dots , \mathrm{dim}\, \cP _n
\, , 
\end{equation}
and set $T_n = U_n^* U_n $. With this notation the solution to
\eqref{eq:8b} is the polynomial   $p_n =
\sum _{k: \lambda _k  \leq n} a_{n,k} \phi _k
\in \cP _n$  with  coefficient vector 
\begin{equation}
  \label{eq:12b}
  a_n = T_n \inv U_n ^* y_n  \, .
\end{equation}
Again, since $\langle T_nc,c\rangle  = \langle U_nc, U_nc  \rangle =
\sum _{k=1}^{L_n} |p(\xkn )|^2 \tkn  $,  the spectrum of
every $T_n$ is contained in the interval $[A,B]$ and we have uniform
 upper bounds for the norms of $T_n$ and $T_n\inv $. 

We now have the analogue of Proposition~\ref{prop1}.

\begin{prop} \label{prop1gen}
Let $\cX $ be a \mz\ family with weights $\tau $.   Let $p_n$ be the solution of the least squares problem
\eqref{eq:8b}. Then 
\begin{align}
  \label{eq:13a}
  \|P_nf - p_n \|_2^2 &\leq A^{-2} B \sum _{k=1}^{L_n} |f(\xkn ) - P_n
                        f(\xkn )|^2 \tkn  \\
  & \leq \tfrac{B^2}{A^2} \|f-P_nf\|^2_\infty \leq
  \kappa ^2  \|f\|^2_{\hs } \phi _\sigma (n)^2 \, . \notag 
\end{align}
\end{prop}

The proof is identical to the proof of Proposition~\ref{prop1}, this
time combined with Lemma~\ref{sobgen}.

\subsection{Approximation of continuous functions from samples}

In this general context the  analog of Theorem~\ref{tm1} read as
follows. 

\begin{tm}
  \label{tm1a}
Assume that $\cX = \{\cX _n :n\in \bN \}$ is a \mz\ family for $M$ with
condition number $\kappa = B/A$ and associated weights $\{\tkn \}$. 

   If $f\in H^\sigma (M)$, then 
  \begin{equation}
    \label{eq:3b}
    \|f-p_n\|_2  \leq \sqrt{1+\kappa ^2} \|f\|_{\hs } \phi _\sigma (n) \, .
  \end{equation}
\end{tm}

\begin{proof}
 The proof is identical to the proof of
  Theorem~\ref{tm1}. This is our main point. We simply use Lemma~\ref{sobgen} and
  Proposition~\ref{prop1gen}, precisely, \eqref{eq:a4} and
  \eqref{eq:13a}  in the  decomposition
$$
\|f- p_n \|_2^2 =   \|f- P_nf \|_2^2 +  \|P_nf - p_n \|_2^2 \, .
$$
\end{proof}

Noting that the error estimates also hold for fixed $n$, we obtain the
following  useful consequence. 
\begin{cor}
Assume that $\{ y_k : k=1, \dots ,L\}\subseteq M$ is a sampling set for $\cP _n$
with weights $\tau _k$,
i.e., for some $A,B>0$ and all $p \in \cP _n$
$$
A \|p\|_2^2 \leq \sum _{k=1}^L |p(y_k)|^2 \tau _k \leq B \|p\|_2^2 \,
.
$$
 For $f\in C(M)$ solve  $q = \mathrm{argmin} _{p\in \cP _n} \sum _{k=1}^L
|f(y_k) - p(y_k)|^2 \tau _k$.  Then 
  \begin{align*}
  \|f-q\|_2 &\leq  \big(1+\big(\tfrac{B}{A}\big)^2 \big)^{1/2} \|f-P_nf\|_\infty  \, .
\end{align*}  
\end{cor}
The corollary gives a possible answer to  how well a given
function on $M$ can be approximated from samples. Again, the statement
highlights the importance of sampling inequalities in approximation
theoretic problems.

\subsection{Quadrature rules}
\label{sec:Quadrat}

Likewise the derivation and  convergence of the quadrature rules is
completely analogous to Theorem~\ref{tm2}.  The reproducing kernel for $\cP _n$ is given by
$k^{(n)}_x(y) = \sum _{k : \lambda _k \leq n} \overline{\phi _k (x)}
\phi _k(y)$.
Then  the \mz\ inequalities  \eqref{eq:1a}  say that every set  $\{\tkn
^{1/2} \, k^{(n)}_{\xkn
} : k=1, \dots , L_n\}$ is a frame for $\cP _n$ with  uniform frame
bounds $A,B>0$. We thus  obtain the dual
frame $e_{n,k} \in \cP _n$ such that every  polynomial $p \in \cP
_n$ can be reconstructed from the samples on $\cX _n$ by     
\begin{equation}
  \label{eq:17a}
  p = \sum _{k=1}^{L_n} \tkn   \langle p , k^{(n)}_{\xkn } \rangle
  S\inv k^{(n)}_{\xkn } = \sum _{k=1}^{L_n} \tkn ^{1/2} p(\xkn ) \,
  e_{n,k} \,  
\end{equation}
Since $1\in \cP _n$ for all $n$, we have again
\begin{equation}
  \label{eq:a1b}
  \sum _{k=1} ^{L_n} |\langle  1, e_{n,k}  \rangle |^2 \leq A\inv \|1\|_2
  = A\inv \, .
\end{equation}
 The weights for the quadrature rules are defined by 
\begin{equation}
  \label{eq:18a}
  w_{n,k} = \tkn ^{1/2} \langle  e_{n,k} ,1 \rangle =  \tkn ^{1/2} \int
  _M e_{n,k}(x) \, d\mu(x) \, , 
\end{equation}
and the corresponding quadrature rule is defined by 
\begin{equation} \label{eq:18bb}
  I_n(f) = \sum _{k=1}^{L_n} f(\xkn ) w_{n,k} \, .
\end{equation}
  Writing $I(f) = \int _M f(x) \, d\mu(x) $, 
the convergence rules for the quadrature \eqref{eq:18bb} can now be
stated as follows.

\begin{tm}
  \label{tm2a}
Let $\cX $ be a \mz\ family on $M$ with weights $\tau $ and let $\{I_n: n \in \bN \}$ be the
 associated sequence of quadrature rules. Assume that $\sigma
>\scrit $.

(i) If $f\in C(M )$, then 
\begin{equation}
  \label{eq:190b}
  |I(f)  - I_n (f)| \leq (1+\sqrt{\kappa })  \inf _{p \in \cP _n}
  \|f-p\|_\infty  \, .
\end{equation}

(ii) If $f\in \hs $ for $\sigma >\scrit $, then 
\begin{equation}
  \label{eq:19b}
  |I(f)  - I_n (f)| \leq (1+\sqrt{\kappa })  \|f\|_{\hs }
  \phi _\sigma (n)  \, . 
\end{equation}
\end{tm}

Error estimates of this type are, of course, 
well-known, see, e.g. ~\cite{CGT11,BCCG14,HS06}.  
The main difference is in the constants: usually these depend mainly
on  the  mesh size of the points $\xkn $, whereas Theorems~\ref{tm1a}
and \ref{tm2a}  involve only the bounds of the \mz\ inequalities. 

Theorems~\ref{tm1a} and \ref{tm2a} are  pure formalism. Their main
insight is  that
approximation theorems from samples are a direct consequence of the
existence of \mz\ families. Thus  the ``real'' and deep question was and still is how to construct a \mz\
family for a given $M$ and orthonormal basis. This is precisely what
is accomplished in \cite{FM10,FM11} under similar assumptions on an
abstract metric measure space.  

\vspace{2mm}

\noindent \textbf{Example.}
  For the torus $\phi _k (x) = e^{2\pi i kx}$ and
$\lambda _k = k$. Then the error function is 
$$
\phi _\sigma (n)^2  = \sum _{|k|>n} (1+k^2)^{-\sigma } \leq 2 \int
_n^\infty x^{-2\sigma } \, dx = \frac{2}{2\sigma -1} n^{-2\sigma +1}
\, .
$$
Consequently  $\phi
_\sigma (n) = (\sigma - 1/2)^{-1/2} n^{-\sigma +1/2} $ in Lemma~\ref{lm1} and
Theorem~\ref{tm1}. 


\subsection{The spectral function and \mz\ families}

Theorem~\ref{tm1a}, as formulated, is almost void of content, as the
error function $\phi _\sigma $ from \eqref{eq:26} depends on the orthonormal basis and the chosen 
$\lambda _k$ in a rather intransparent manner. With the interpretation
of $(\phi _k,
\lambda _k)$ as the eigenvalues and eigenfunctions of a positive
unbounded operator, the  spectral theory of partial differential
operators suggests a suitable condition to elaborate  the error
function $\phi _\sigma $ further.

\begin{definition}
  We say the orthonormal basis $\{\phi _k:k\in \bN \}$ and the
  eigenvalues $\{\lambda _k: k\in \bN \}$ satisfy Weyl's law, if there
  exist constants $d=d_{\phi,\lambda }>0$ and $C= C_{\phi,\lambda
  }>0$, such that
  \begin{equation}
    \label{eq:weyl}
    \sum _{k: \lambda _k \leq n} |\phi _k(x)|^2 \leq C n^d \, .
  \end{equation}
\end{definition}
Integrating   over $M$, \eqref{eq:weyl} implies the eigenvalue count
\begin{equation}
  \label{eq:weyl2}
\# \{ k: \lambda _k \leq n\}  \leq C n^d \, .  
\end{equation}
In the spectral  theory of partial differential operators or pseudodifferential
operators the function $  \sum _{k: \lambda _k \leq n} |\phi _k(x)|^2$
is called the spectral function, and \eqref{eq:weyl2} is Weyl's law
for the count of eigenvalues~\cite{Gar53,hormander3,Shubin91}. In the theory
of orthogonal polynomials  the function  $ (\sum _{k\leq n} |\phi
_k(x)|^2)\inv $ is called the Christoffel function and  plays a
central role in the investigation of orthogonal
polynomials~\cite{Nev86}. In this context assumption \eqref{eq:weyl}
says that  the Christoffel function along a subsequence  determined by
the  $\lambda _k$'s is bounded
polynomially from below.  

Under the assumption of Weyl's law we can determine the critical
exponent and the asymptotics of the error function precisely.  One may
also say that Weyl's law implies the correct version of the Sobolev
embedding. 
\begin{prop}
  \label{critval}
Let $\{\phi _k:k\in \bN \}$ be an orthonormal basis for $L^2(M,\mu )$
and $\lambda _k \geq 0$ be a non-decreasing sequence with $\lim _{k\to
  \infty } \lambda _k = \infty $. Assume that $(\phi _k, \lambda _k)$
satisfies Weyl's law \eqref{eq:weyl} with exponent $d$. Then the critical value is $\scrit = d/2$, i.e., if $\sigma
  >d/2$, then
  $$
  C_\sigma ^2 = \sup _{x\in M}  \sum _{k=1}^\infty |\phi _k(x)|^2
  (1+\lambda _k^2)^{-\sigma } <\infty \, . 
  $$
Moreover, the error function is
  \begin{equation}
    \label{eq:a5}
    \phi _\sigma (n) = C_\sigma  \, n^{-\sigma +d/2}
  \end{equation}
\end{prop}
\begin{proof}
  We only show \eqref{eq:a5}. Choose $M \in \bN$, such that $2^{M}
  \leq n \leq 2^{M+1}$. We  split the sum defining $\phi _\sigma $
  into dyadic blocks and use \eqref{eq:weyl} as follows:  
  \begin{align*}
  \sum _{\lambda _k >n}^\infty |\phi _k(x)|^2
    (1+\lambda _k^2)^{-\sigma }     & \leq   \sum _{n=M}^\infty \sum _{k: 2^n \leq
    \lambda _k <  2^{n+1}} |\phi _k(x)|^2   (1+\lambda _k^2)^{-\sigma } \\
    &\leq  \sum _{n=M}^\infty 2^{-2n \sigma } \sum _{k: 2^n \leq
      \lambda _k <  2^{n+1}} |\phi _k(x)|^2  \\
    &\leq   C \sum _{n=M} ^\infty 2^{-2n \sigma } 2^{d(n+1)} 
     =  C 2^d\sum _{n=M} ^\infty 2^{n(d-2\sigma )} \\
    &= C \frac{2^d}{ (1-2^{d-2\sigma } )} \,  2^{M(d-2\sigma ) }  \leq
      C_\sigma n^{-2\sigma +d}  \, ,
  \end{align*}
with convergence precisely for $\sigma > d/2$.   
\end{proof}

\noindent \textbf{Example.} Let $-\Delta $ be the Laplacian on a
bounded domain $M\subseteq \rd $ with $C^\infty$-boundary and $\phi
_k$ be its  eigenfunctions, 
i.e., $-\Delta \phi _k =
\lambda _k^2 \phi _k$. 
Then the $(\phi _k,
\lambda _k)$'s satisfy Weyl's law with exponent $d$ and the constant is
roughly $C= \mathrm{vol}\,  (M)$.  Similar statements hold for the
Laplace-Beltrami operator on a compact Riemannian manifold and for
general elliptic partial differential operators, see~\cite{Gar53,hormander68,hormander3,Shubin91}

Weyl's law plays an important role in  the theory and construction of \mz\
families in metric measure spaces by Filbir and
Mhaskar~\cite{FM10,FM11}. In particular, they show that Weyl's law is
equivalent to Gaussian estimates for the heat kernel. A weaker form of
Weyl's law is equivalent to a sampling inequality~\cite{Pes16}.





\vspace{ 2mm}

\noindent\textbf{Concluding remarks.} The simplicity of the proofs
lead to conceptual insights into the role of \mz\ families, but the
results are certainly limited.

(i) The proofs work only for $p=2$.  \mz\ families with respect to
general $p$-norms seem to require different techniques. 

(ii) The weights for the  quadrature rules are not necessarily
positive. So far, positive weights  are obtained only with special
constructions from \mz\
families with $p=1$ and sufficiently dense sampling on each
level~\cite{FM10}.

(iii) The existence of \mz\ families has been established on many
level of generality~\cite{FM11,FFP16} through the construction of
sufficiently fine meshes on each level. So far necessary density
conditions have been found only for \mz\ on compact Riemannian
manifolds~\cite{OP12}. It would be interesting to extend the scope of
the necessary conditions to the conditions used in Section~3 and to
compare the densities of the existing \mz\ families to the necessary
density conditions.


\begin{thebibliography}{10}

\bibitem{AC19}
B.~Adcock and J.~M. Cardenas.
\newblock Optimal sampling strategies for multivariate function approximation
  on general domains.
\newblock 2019.
\newblock arXiv:1908.01249.

\bibitem{AH12}
  B.~ Adcock, and A.C.~Hansen.
  \newblock A generalized sampling theorem for stable reconstructions in arbitrary
  bases.
  \newblock {em J. Fourier Anal. Appl.}  18(4):685--716, 2012.

  
  
\bibitem{AHP13}
B.~Adcock, A.~C. Hansen, and C.~Poon.
\newblock Beyond consistent reconstructions: optimality and sharp bounds for
  generalized sampling, and application to the uniform resampling problem.
\newblock {\em SIAM J. Math. Anal.}, 45(5):3132--3167, 2013.




\bibitem{ABC19}
B.~Arras, M.~Bachmayr, and A.~Cohen.
\newblock Sequential sampling for optimal weighted least squares approximations
  in hierarchical spaces.
\newblock {\em SIAM J. Math. Data Sci.}, 1(1):189--207, 2019.

\bibitem{AT17}
A.~P. Austin and L.~N. Trefethen.
\newblock Trigonometric interpolation and quadrature in perturbed points.
\newblock {\em SIAM J. Numer. Anal.}, 55(5):2113--2122, 2017.

\bibitem{BG04}
R.~F. Bass and K.~Gr\"{o}chenig.
\newblock Random sampling of multivariate trigonometric polynomials.
\newblock {\em SIAM J. Math. Anal.}, 36(3):773--795, 2004/05.

\bibitem{BH90}
J.~J. Benedetto and W.~Heller.
\newblock Irregular sampling and the theory of frames. {I}.
\newblock {\em Note Mat.}, 10(suppl. 1):103--125 (1992), 1990.
\newblock Dedicated to the memory of Professor Gottfried K\"othe.

\bibitem{BCCG14}
L.~Brandolini, C.~Choirat, L.~Colzani, G.~Gigante, R.~Seri, and G.~Travaglini.
\newblock Quadrature rules and distribution of points on manifolds.
\newblock {\em Ann. Sc. Norm. Super. Pisa Cl. Sci. (5)}, 13(4):889--923, 2014.

\bibitem{Chr16}
O.~Christensen.
\newblock {\em An introduction to frames and {R}iesz bases}.
\newblock Applied and Numerical Harmonic Analysis. Birkh\"{a}user/Springer,
  [Cham], second edition, 2016.

\bibitem{CM17}
A.~Cohen and G.~Migliorati.
\newblock Optimal weighted least-squares methods.
\newblock {\em SMAI J. Comput. Math.}, 3:181--203, 2017.

\bibitem{CGT11}
L.~Colzani, G.~Gigante, and G.~Travaglini.
\newblock Trigonometric approximation and a general form of the {E}rd\"{o}s
  {T}ur\'{a}n inequality.
\newblock {\em Trans. Amer. Math. Soc.}, 363(2):1101--1123, 2011.

\bibitem{DTT19}
  F. Dai, A. Prymak, V.N. Temlyakov, and S. Tikhonov.
  \newblock Integral norm discretization and related
  problems.
  \newblock Preprint.  arxiv.org/abs/1807.01353
  
\bibitem{DS52}
R.~J. Duffin and A.~C. Schaeffer.
\newblock A class of nonharmonic {F}ourier series.
\newblock {\em Trans. Amer. Math. Soc.}, 72:341--366, 1952.

\bibitem{EGO17}
M.~Ehler, M.~Graef, and C.~J. Oates.
\newblock Optimal monte carlo integration on closed manifolds.
\newblock {\em Preprint, arXiv:1707.04723}.

\bibitem{Erd99}
T.~Erd\'{e}lyi.
\newblock Notes on inequalities with doubling weights.
\newblock {\em J. Approx. Theory}, 100(1):60--72, 1999.

\bibitem{FFP16}
H.~G. Feichtinger, H.~F\"uhr, and I.~Z. Pesenson.
\newblock Geometric space-frequency analysis on manifolds.
\newblock {\em J. Fourier Anal. Appl.}, 22(6):1294--1355, 2016.

\bibitem{FGtp94}
H.~G. Feichtinger and K.~Gr\"{o}chenig.
\newblock Theory and practice of irregular sampling.
\newblock In {\em Wavelets: mathematics and applications}, Stud. Adv. Math.,
  pages 305--363. CRC, Boca Raton, FL, 1994.

\bibitem{fgs95}
H.~G. Feichtinger, K.~Gr{\"o}chenig, and T.~Strohmer.
\newblock Efficient numerical methods in non-uniform sampling theory.
\newblock {\em Numer. Math.}, 69(4):423--440, 1995.

\bibitem{FM10}
F.~Filbir and H.~N. Mhaskar.
\newblock A quadrature formula for diffusion polynomials corresponding to a
  generalized heat kernel.
\newblock {\em J. Fourier Anal. Appl.}, 16(5):629--657, 2010.

\bibitem{FM11}
F.~Filbir and H.~N. Mhaskar.
\newblock Marcinkiewicz-{Z}ygmund measures on manifolds.
\newblock {\em J. Complexity}, 27(6):568--596, 2011.

\bibitem{Gar53}
L.~Garding.
\newblock On the asymptotic distribution of the eigenvalues and eigenfunctions
  of elliptic differential operators.
\newblock {\em Math. Scand.}, 1:237--255, 1953.

\bibitem{Gro93c}
K.~Gr\"ochenig.
\newblock A discrete theory of irregular sampling.
\newblock {\em Linear Algebra Appl.}, 193:129--150, 1993.

\bibitem{Gro99}
K.~Gr{\"o}chenig.
\newblock Irregular sampling, {T}oeplitz matrices, and the approximation of
  entire functions of exponential type.
\newblock {\em Math. Comp.}, 68(226):749--765, 1999.

\bibitem{Gro01b}
K.~Gr{\"o}chenig.
\newblock Non-Uniform Sampling in Higher Dimensions: From Trigonometric
Polynomials to Band-Limited Functions.
\newblock In  ``Modern Sampling Theory: Mathematics and Applications'', Chap.~7,
 pp.~155 -- 171, J.~J.~Benedetto, P.~Ferreira,
eds.,  Birkh\"auser, Boston, 2001.

\bibitem{HS06}
K.~Hesse and I.~H. Sloan.
\newblock Cubature over the sphere {$S^2$} in {S}obolev spaces of arbitrary
  order.
\newblock {\em J. Approx. Theory}, 141(2):118--133, 2006.

\bibitem{hormander68}
L.~H\"{o}rmander.
\newblock The spectral function of an elliptic operator.
\newblock {\em Acta Math.}, 121:193--218, 1968.

\bibitem{hormander3}
L.~H{\"o}rmander.
\newblock {\em The analysis of linear partial differential operators. {III}},
  volume 274 of {\em Grundlehren der Mathematischen Wissenschaften [Fundamental
  Principles of Mathematical Sciences]}.
\newblock Springer-Verlag, Berlin, 1994.
\newblock Pseudo-differential operators, Corrected reprint of the 1985
  original.



  
\bibitem{Lub98}
D.~S. Lubinsky.
\newblock Marcinkiewicz-{Z}ygmund inequalities: methods and results.
\newblock In {\em Recent progress in inequalities ({N}i\v{s}, 1996)}, volume
  430 of {\em Math. Appl.}, pages 213--240. Kluwer Acad. Publ., Dordrecht,
  1998.

\bibitem{Lub99}
D.~S. Lubinsky.
\newblock On converse {M}arcinkiewicz-{Z}ygmund inequalities in {$L_p$},
  {$p>1$}.
\newblock {\em Constr. Approx.}, 15(4):577--610, 1999.

\bibitem{MM08}
M.~Maggioni and H.~N. Mhaskar.
\newblock Diffusion polynomial frames on metric measure spaces.
\newblock {\em Appl. Comput. Harmon. Anal.}, 24(3):329--353, 2008.

\bibitem{Mar07}
J.~Marzo.
\newblock Marcinkiewicz-{Z}ygmund inequalities and interpolation by spherical
  harmonics.
\newblock {\em J. Funct. Anal.}, 250(2):559--587, 2007.

\bibitem{MOC10}
J.~Marzo and J.~Ortega-Cerd\`a.
\newblock Equidistribution of {F}ekete points on the sphere.
\newblock {\em Constr. Approx.}, 32(3):513--521, 2010.

\bibitem{MP14}
J.~Marzo and B.~Pridhnani.
\newblock Sufficient conditions for sampling and interpolation on the sphere.
\newblock {\em Constr. Approx.}, 40(2):241--257, 2014.

\bibitem{MT00}
G.~Mastroianni and V.~Totik.
\newblock Weighted polynomial inequalities with doubling and {$A_\infty$}
  weights.
\newblock {\em Constr. Approx.}, 16(1):37--71, 2000.

\bibitem{Mh18}
H.~N. Mhaskar.
\newblock {\em Approximate Quadrature Measures on Data-Defined Spaces}, pages
  931--962.
\newblock Springer International Publishing, Cham, 2018.

\bibitem{MNW01}
H.~N. Mhaskar, F.~J. Narcowich, and J.~D. Ward.
\newblock Spherical {M}arcinkiewicz-{Z}ygmund inequalities and positive
  quadrature.
\newblock {\em Math. Comp.}, 70(235):1113--1130, 2001.

\bibitem{Nev86}
P.~Nevai.
\newblock G\'{e}za {F}reud, orthogonal polynomials and {C}hristoffel functions.
  {A} case study.
\newblock {\em J. Approx. Theory}, 48(1):3--167, 1986.

\bibitem{OP12}
J.~Ortega-Cerd\`a and B.~Pridhnani.
\newblock Beurling-{L}andau's density on compact manifolds.
\newblock {\em J. Funct. Anal.}, 263(7):2102--2140, 2012.

\bibitem{OS07}
J.~Ortega-Cerd\`a and J.~Saludes.
\newblock Marcinkiewicz-{Z}ygmund inequalities.
\newblock {\em J. Approx. Theory}, 145(2):237--252, 2007.

\bibitem{Pes16}
I.~Z. Pesenson.
\newblock Shannon sampling and weak {W}eyl's law on compact {R}iemannian
  manifolds.
\newblock In {\em Analysis and partial differential equations: perspectives
  from developing countries}, volume 275 of {\em Springer Proc. Math. Stat.},
  pages 207--218. Springer, Cham, 2019.

\bibitem{PST01}
D.~Potts, G.~Steidl, and M.~Tasche.
\newblock Fast {F}ourier transforms for nonequispaced data: a tutorial.
\newblock In {\em Modern sampling theory}, Appl. Numer. Harmon. Anal., pages
  247--270. Birkh\"{a}user Boston, Boston, MA, 2001.

\bibitem{RS97}
K.~V. Runovski and W.~Sickel.
\newblock Marcinkiewicz-{Z}ygmund-type inequalities---trigonometric
  interpolation on non-uniform grids and unconditional {S}chauder bases in
  {B}esov spaces on the torus.
\newblock {\em Z. Anal. Anwendungen}, 16(3):669--687, 1997.

\bibitem{Shubin91}
M.~A. Shubin.
\newblock {\em Pseudodifferential Operators and Spectral Theory}.
\newblock Springer-Verlag, Berlin, second edition, 2001.
\newblock Translated from the 1978 Russian original by Stig I. Andersson.

\bibitem{SZ04}
S.~Smale and D.-X. Zhou.
\newblock Shannon sampling and function reconstruction from point values.
\newblock {\em Bull. Amer. Math. Soc. (N.S.)}, 41(3):279--305, 2004.

\bibitem{Sunw02}
W.~Sun and X.~Zhou.
\newblock Reconstruction of band-limited functions from local averages.
\newblock {\em Constr. Approx.}, 18(2):205--222, 2002.

\bibitem{TW14}
L.~N. Trefethen and J.~A.~C. Weideman.
\newblock The exponentially convergent trapezoidal rule.
\newblock {\em SIAM Rev.}, 56(3):385--458, 2014.

\end{thebibliography}

\end{document}